\documentclass[twoside,11pt]{amsart}
\usepackage[active]{srcltx} 
\usepackage[all]{xy}
\CompileMatrices
\usepackage{amsfonts}
\usepackage{amssymb}
\usepackage{amsthm}
\usepackage{amsmath}
\usepackage{ifthen}
\usepackage{url}
\usepackage{scalerel}
\usepackage{multicol}
\usepackage{multirow}
\usepackage{stmaryrd}
\usepackage[ngerman,shadow,color=blue!30,textsize=footnotesize,colorinlistoftodos]{todonotes}
 \newtheorem{theorem}{Theorem}[section]

 \newtheorem{lemma}[theorem]{Lemma}
  \newtheorem{remark}[theorem]{Remark}
  
 \newtheorem{conjecture}[theorem]{Conjecture}

\theoremstyle{definition}
\newtheorem{definition}[theorem]{Definition}

\newtheorem{example}[theorem]{Example}

\DeclareMathOperator{\Gal}{Gal}            

\DeclareMathOperator{\ab}{ab}
\DeclareMathOperator{\Pl}{Pl}

\DeclareMathOperator{\lc}{lc}

\newcommand{\Q}{\mathbb{Q}}

\newcommand{\Z}{\mathbb{Z}}
\newcommand{\N}{\mathbb{N}}

\newcommand{\Zl}{\mathbb{Z}_{\ell}}

\newcommand{\LL}{\Lambda}

\newcommand{\M}{\mathfrak{M}}

\newcommand{\p}{\mathfrak{p}}

\newcommand{\x}{\chi}

\newcommand{\Rc}{\mathcal{R}}
\newcommand{\Uc}{\mathcal{U}}

\newcommand{\Clog}[1]{\widetilde{C\ell}_{#1}}

\newcommand{\mut}{\widetilde{\mu}}
\newcommand{\lat}{\widetilde{\lambda}}
\newcommand{\et}{\widetilde{e}}

\newcommand{\isom}{\simeq}

\def\sectionnam{\@empty}
\def\subsectionnam{\@empty}

\begin{document}

\title[Topological behaviour of logarithmic invariants]{Topological behaviour of logarithmic invariants}%

\author{Jos\'e-Ibrahim Villanueva-Guti\'errez}%

\address{Universit\"at Heidelberg\\ Mathematisches Institut\\
Im Neuenheimer Feld 205 \\69120 Heidelberg, Germany.} \email{jgutierrez@mathi.uni-heidelberg.de}
\urladdr{https://www.mathi.uni-heidelberg.de/\textasciitilde jgutierrez}
\thanks{This paper was written partially under the financial support of CONACYT (The Mexican Council of Science and Technology) as part of the author's PhD project.}
\date{\today}%
\keywords{Iwasawa invariants, logarithmic class groups, logarithmic invariants} 
\subjclass{11R23}

\begin{abstract}
Let $\ell$ be a rational prime number and $K$ a number field. We prove that the logarithmic module $X_{d}$ attached to a $\mathbb{Z}_{\ell}^{d}$-extension $K_{d}$ of $K$ is a noetherian $\Lambda_{d}$-module. Moreover, under the Gross-Kuz'min conjecture we prove that it is also torsion. We exploit this fact to deduce local and global information of the logarithmic invariants $\tilde{\mu}$ and $\tilde{\lambda}$ of $\mathbb{Z}_{\ell}$-extensions. 
\end{abstract} 

\maketitle
\thispagestyle{empty}

\section{Introduction}

Let $\ell$ be a prime number and $K$ a number field. A milestone theorem in Iwasawa theory states that the exponent $e_{n}$ of the $\ell$-part of the class group attached to a finite layer $K_{n}$ of a $\Zl$-extension $K_{\infty}$ of $K$ is given by $e_{n}=\mu \ell^{n}+\lambda n + \nu$ for $\mu,\lambda\geq 0$ and $\nu$ integers, for $n$ big enough.

Great efforts have been made in many directions to understand the behaviour of the invariants and to compute them explicitly. In \cite{Greenberg73IwasawaInvariants} Greenberg considers the set $\Delta(K)$ of all $\Zl$-extensions of a number field $K$ and proves that the $\mu$ invariant is locally bounded in some dense subset of $\Delta(K)$. It is well known that $\Delta(K)$ is non-empty. He also proves that $\lambda$ is locally bounded in a neighbourhood of a $\Zl$-extension $L$ with zero $\mu(L/K)$ invariant. These results rely heavily on the fact that the Galois group $\Gal(H_{d}/K_{d})$ of the maximal unramified $\ell$-extension of the compositum $K_{d}$ of the $\Zl$-extensions of $K$ is a noetherian torsion module over the Iwasawa algebra $\Zl[[\Gal(K_{d}/K)]]$, which is the profinite group algebra of $\Gal(K_{d}/K)\isom \Zl^{d}$ for some $d$. 

In the same work Greenberg raised some questions concerning the global and maximal behaviour of Iwaswa invariants. Independently Baba\u{\i}cev and Monsky proved that the $\mu$-invariant is bounded in $\Delta(K)$ \cite{Babauicev80,Monsky81SomeInvariantsZpdExt}. Using a finer topology and generalizing a theorem of Fukuda, Kleine proved that the Iwasawa invariants are locally maximal \cite{Kleine17}.

In this work we study these results from the logarithmic arithmetic point of view. That is, our objects will be logarithmic as in \cite{Jaulent94,Villanueva18}. 

In \textit{loc. cit.} we proved that Iwasawa's theorem holds in our setting. That is, the exponent $\widetilde{e}_{n}$ of the $\ell$-part of the logarithmic class group attached to $K_{n}$ is given by $\widetilde{e}_{n}=\mut \ell^{n}+\lat n + \tilde{\nu}$ for $\mut,\lat\geq 0$ and $\tilde{\nu}$ integers, for $n$ big enough.

Analogies between the logarithmic arithmetic and classical arithmetic allows us to reproduce some of the proofs in our setting. Nonetheless some preparations must be made. Also, we like to highlight the parts where the Gross-Kuz'min conjecture is used. 

The structure of the paper goes as follows. First we start by recalling some results about $\Zl^{d}$-extensions, $\LL_{d}$-modules and logarithmic arithmetic. Further details on the Iwasawa theory of $\Zl^{d}$-extensions and $\LL_{d}$-modules can be found in \cite{Greenberg73IwasawaInvariants} and \cite[Ch. V]{Neukirch&Schmidt&Wingberg08}. For information on logarithmic arithmetic we invite the reader to consult \cite{Jaulent94,Villanueva18}. At the end of the section we recall what we know about the cyclotomic $\Zl$-extension and we introduce the topologies which will be used in this work.

Then in section \ref{Sec:LogMod} we prove that the logarithmic module $X_{d}$ asociated to a $\Zl^{d}$-extension is noetherian and prove that assuming the Gross-Kuz'min conjecture it is torsion. 

In section \ref{Sec:LGM} we prove some results on the local boundedness of the Iwasawa logarithmic invariants $\mut$ and $\lat$ of a $\Zl$-extension of $K$ with respect to Greenberg's topology and with respect to log-Kleine's topology. Finally, we also prove that $\mut$ is globally bounded on $\Delta(K)$.

\textbf{Aknowledgements:} I would like to thank Antonio Lei who introduced me to general $p$-adic Lie groups. Thanks to Michael F\"utterer, Katharina H\"ubner, S\"oren Kleine and Oliver Thomas for the useful discussions. 

\section{Preliminaries}

\subsection{$\Zl^{d}$-extensions}

Let $K$ be a number field and let $\ell$ be a rational prime number. We denote by $r$ and $c$ the number of real and pairs of complex embeddings of $K$, respectively. A $\Zl^{d}$-extension $K_{d}$ of $K$ is an infinite Galois extension such that $\Gal(K_{d}/K)$ is isomorphic to the product of $d$ copies of the $\ell$-adic integers $\Zl$. Every number field admits at least one such an extension in the special case $d=1$, namely the cyclotomic extension (e.g. \cite[\S 7.2]{Washington97}). 

Two $\Zl$-extensions $K_{\infty}/K$ and $K'_{\infty}/K$ are said to be independent over $K$ if $K_{\infty}\cap K'_{\infty}=K$. Let $K^{(1)}_{\infty},\ldots,K^{(d)}_{\infty}$ be $d$ independent $\Zl$-extensions of $K$. The composite $K_{d}:=K^{(1)}_{\infty}\cdots K^{(d)}_{\infty}$ is a Galois abelian extension with $\Gamma_{d}:=\Gal(K_{d}/K)\isom \Zl^{d}$. Endowing $\Gamma_{d}$ with the product topology induced by $\Zl$ we choose topological generators $\gamma_{1},\ldots,\gamma_{d}$ of $\Gamma_{d}$.

We have the following bounds (loc.cit. \S 5.5) for $d$, the number of independent $\Zl$-extensions of a number field $K$: $$c+1\leq d \leq [K:\Q].$$ 
More precisely it is the number of pairwise independent $\Zl$-extensions of $K$. The equality $d=c+1$ holds when $K/\Q$ is abelian and in some other cases. Leopoldt's conjecture asserts that in fact the equality $d=c+1$ holds for every number field $K$. From now on, we will not worry about the precise value of $d$. 

\begin{example} Let $K$ be a quadratic imaginary number field and let $M_{\ell}$ be the maximal $\ell$-abelian $\ell$-ramified extension of $K$. Then $\Gal(M_{\ell}/K)=\Zl^{2}\times \text{finite group}$. If $\ell$ does not divide the class group of $K$, then $\Gal(M_{\ell}/K)=\Zl^{2}$. Hence, $M_{\ell}$ contains two independent $\Zl$-extensions, namely the cyclotomic and the anti-cyclotomic $\Zl$-extension.
\end{example}

\subsection{$\LL_{d}$-modules}

Let $\LL_{d}:=\LL_{d}(\Gamma_{d})=\varprojlim_{U\unlhd \Gamma_{d}} \Zl[\Gamma_{d}/U]$  (or simply denoted $\LL$ when $d=1$ and no confusion arises) be the Iwasawa algebra of $\Gamma_{d}$. The mapping $\gamma_{i}\mapsto T_{i}+1$ defines an algebra isomorphism $\LL_{d}\isom \Zl\llbracket T_{1},\ldots,T_{d} \rrbracket $. Hence, $\LL_{d}$ is a local regular ring of dimension $d+1$, with maximal ideal $\M=(T_{1},\ldots,T_{d},\ell)$.

\begin{definition} We say that two modules $\LL_{d}$-modules $M$ and $N$ are pseudo-isomorphic if there exists a $\LL_{d}$-module homomorphism $f:M\rightarrow N$ such that localisation at every prime ideal $\p\in\LL_{d}$ of height 1 induces an isomorphism $f_{\p}:M_{\p}\rightarrow N_{\p}$. A module is pseudo-null if it is pseudo-isomorphic to 0.
\end{definition}

\begin{remark} This definition agrees with the one given in \cite[\S 13.2]{Washington97} for $d=1$, since the localisation of a finite module at a prime ideal of height 1 in $\Zl\llbracket T \rrbracket $ is 0.
\end{remark} 

The following theorem is a special case of the structure theorem for noetherian torsion modules over local regular rings.

\begin{theorem}[Structure theorem]\label{thm:structure} Every noetherian torsion $\LL_{d}$-module $M$ is pseudo-isomorphic to a direct sum as follows
$$M \sim \bigoplus_{i} \LL_{d}/P_{i}^{r_{i}}, $$
where $P_{i}$ are height 1 prime ideals of $\LL_{d}$ and $r_{i}\in\N$.
\end{theorem}

In the case $d=1$ the structure theorem can be refined as follows. Every noetherian torsion $\LL$-module $M$ is pseudo isomorphic to an elementary module 
$$E:=\left(\bigoplus\limits_{i=1}^{s}\LL/\ell^{m_{i}}\LL\right)
\oplus\left(\bigoplus\limits_{j=1}^{t}\LL/P_{j}\LL\right)$$
with $P_{j}$ distinguished polynomials ordered by divisibility: $P_{1}|P_{2}|\ldots|P_{t}$. Then 
$$\x=\prod_{i=1}^{s}\ell^{m_{i}}\prod_{j=1}^{t}P_{j}, \	\	\	\mu=\sum_{i=1}^{s} m_{i} \	\	\	\textrm{ and }\	\	\	\lambda=\sum_{j=1}^{t}\deg(P_{j});$$
$\x$ is called  the characteristic polynomial of $M$, and $\mu$ and $\lambda$ its structural invariants. 
 
In the case $M=\Gal(K_{\infty}^{\lc}/K_{\infty})$ is the logarithmic $\LL$-module corresponding to a $\Zl$-extension $K_{\infty}$, we denote its structural invariants $\mut(K_{\infty})$ and $\lat(K_{\infty})$ (see \S \ref{Sec:LogMod}).  
 
Consider an epimorphism  $\pi:\LL_{d}\rightarrow \LL$ and a noetherian torsion $\LL_{d}$-module $M$. We denote $\mathfrak{A}_{\pi}$ the kernel of such an epimorphism. The quotient 
$$ M_{\pi}= M/(\mathfrak{A}_{\pi}\cdot M) $$ 
is clearly a noetherian $\LL$-module. 

\subsection{Logarithmic ramification} Let $L/K$ be a finite extension, $\mathfrak{P}$ and $\p$ be places of $L$ and $K$, respectively, with $\mathfrak{P}|\p$ and $\p|p$. We write $K_{\p}$ (resp. $L_{\mathfrak{P}}$) for the completion of $K$ at $\p$ (resp. $L$ at $\mathfrak{P}$). We denote $\widehat{\Q_{p}^{c}}$ the composite of all $\Z_{q}$-cyclotomic extensions of $\Q_{p}$ for every prime $q$.  

The logarithmic ramification index and the logarithmic inertia degree are defined as 
$$\tilde{e}_{\mathfrak{P}|\p}:=[L_{\mathfrak{P}}:\widehat{\Q_{p}^{c}}K_{\p}\cap L_{\mathfrak{P}}] \	\	\	\text{ and }\	\	\	\tilde{f}_{\mathfrak{P}|\p}:=[\widehat{\Q_{p}^{c}}K_{\p}\cap L_{\mathfrak{P}}:K_{\p}].$$ 
They satisfy the usual multiplicative relations in towers \cite[Thm. 1.4]{Jaulent94}, i.e. 
$$[L_{\mathfrak{P}}:K_{\p}]=\tilde{e}_{\mathfrak{P}|\p}\tilde{f}_{\mathfrak{P}|\p}, \	\	\	 \tilde{e}_{\mathfrak{P}}=\tilde{e}_{\mathfrak{P}|\p}
\tilde{e}_{\p}\	\	\	\text{ and }\	\	\	\tilde{f}_{\mathfrak{P}}=\tilde{f}_{\mathfrak{P}|\p}
\tilde{f}_{\p}.$$ We say that $\p$ is logarithmically ramified (log-ramified) if $\tilde{e}_{\mathfrak{P}|\p}>1$, otherwise we say $\p$ is logarithmically unramified.

If $L/K$ is an abelian $\ell$-extension, the ramification index is exactly the order of some subgroup of $\Gal(L/K)$. We recall briefly how this is done and we invite the reader to consult \cite{Jaulent94,Villanueva18} for further references. By $\ell$-adic class field theory we can identify the Galois group of the maximal abelian $\ell$-extension of $K_{\p}$ with the inverse limit $\Rc_{\p}:=\varprojlim K_{\p}^{\times}/K_{\p}^{\times,\ell^{n}}$. Hence $\Rc_{\p}$ corresponds to the decomposition subgroup of $\p$ in $K^{\ab}$, the maximal abelian $\ell$-extension of $K$. By taking the logarithmic valuation $\tilde{v}_{\p}$ \cite[\S 2]{Villanueva18}, we can decompose $\Rc_{\p}$ as a product of $\Zl$-modules $\tilde{\pi}_{\p}^{\Zl}\widetilde{U}_{\p}$, where $\widetilde{U}_{\p}$ is the kernel of $\tilde{v}_{\p}$. We call the image of $\widetilde{U}_{\p}$ in $\Gal(L/K)$ the logarithmic inertia subgroup $\tilde{I}_{\p}$. 

The locally cyclotomic extension $K^{\lc}$ of $K$ corresponds to the maximal abelian $\ell$-extension of $K$ which is logarithmically unramified (everywhere). Under the Galois correspondence it corresponds to the abelian extension fixed by the image in $\Gal(K^{\ab}/K)$ of the product $\prod \widetilde{U}_{\p}$ over the finite places of $K$ of the logarithmic unit subgroups $\widetilde{U}_{\p}$. 

\begin{remark}\label{rem:logram} Unlike the classical case, the extension $K^{\lc}$ is infinite since it contains the cyclotomic $\Zl$-extension $K^{\text{c}}$ of $K$.
\end{remark}

We will use the following result.

\begin{lemma}\label{lem:logramlram} Let $E/F$ be a Galois $\ell$-extension which is logarithmically unramified. Then $E/F$ is unramified outside $\ell$.
\end{lemma} 
\begin{proof} By definition we have $\et_{\mathfrak{P}|\p}=1$ for all finite places $\mathfrak{P}$ of $E$. Take a place $\p|p$ with $p\neq \ell$. By \cite[Thm. 1.4]{Jaulent94} we have $v_{q}(\et_{\mathfrak{P}|\p})=v_{q}(e_{\mathfrak{P}|\p})=0$ for every $q\neq p$. Finally, suppose that $v_{p}(e_{\mathfrak{P}|\p})\neq 0$, then $p|e_{\mathfrak{P}|\p}$ which implies that $p|[E:F]$. \end{proof}

\subsection{The logarithmic class group}

The group $\Clog{K}^{*}$ of logarithmic classes of arbitrary degree is the Galois group $\Gal(K^{\lc}/K)$ of the maximal abelian $\ell$-extension of $K$ which splits completely over the cyclotomic $\Zl$-extension $K^{c}$ of $K$. As stated before $K^{\lc}$ is the maximal abelian logarithmically unramified $\ell$-extension of $K$. The logarithmic class group $\Clog{K}$ of a number field $K$ corresponds to the Galois group $\Gal(K^{\lc}/K^{\text{c}})$.

The logarithmic class group is conjectured to be finite for every number field $K$ and prime $\ell$. 

\begin{conjecture}[Gross-Kuz'min conjecture] The group $\Clog{K}^{*}$ is a $\Zl$-module of rank 1. Equivalently, the logarithmic class group $\Clog{K}$ is a finite group. 
\end{conjecture}

The conjecture holds if $K/\Q$ is an abelian extension and has been proved in some other cases. Additionally it can be verified numerically for pairs $(K,\ell)$ thanks to the implementation of the logarithmic class group computation in PARI/GP \cite{Belabas&Jaulent16}. 

\begin{lemma}\label{lem:logramZlext} Let $K_{d}/K$ be a $\Zl^{d}$-extension with $d\geq2$. Suppose that $K$ satisfies the Gross-Kuz'min conjecture. Then every prime $\p$ of $K$ not dividing $\ell$ is (log)-unramified in $K_{d}/K$ and at least one prime $\p|\ell$ (log)-ramifies in $K$.
\end{lemma}
\begin{proof}
Let $\p$ be a place of $K$ not dividing $\ell$. In this case the logarithmic valuation $\widetilde{v}_{\p}$ and the usual valuation $v_{\p}$ coincide \cite[\S 2]{Villanueva18}. The kernel $\widetilde{\Uc}_{\p}$ of the valuation map $\Rc_{\p}\rightarrow \Zl$, is a finite $\ell$-group which corresponds, by $\ell$-adic class field theory, to the logarithmic inertia subgroup $\tilde{I}_{\p}$ of the Galois group $\Gal(K^{\ab}/K)$ of the maximal abelian $\ell$-extension of $K$.  Therefore, the image of $\tilde{I}_{\p}$ in $\Gal(K_{d}/K)$ must be trivial, since $Gal(K_{d}/K)$ has no $\ell$-torsion. Using Lemma \ref{lem:logramlram} we get the same result for classical ramification. 

By the finiteness of the class group, we know that at least one place $\p|\ell$ must ramify. This is also the case for logarithmic ramification. If $K_{d}/K$ is logarithmically unramified, then $K_{d}$ is contained in $K^{\lc}$ which is a contradiction to the Gross-Kuz'min conjecture. \end{proof}

Notice that Remark \ref{rem:logram} implies that the preceding lemma is no longer true when $d=1$.

\subsection{The cyclotomic case}\label{sec:thecyclocase}

Let $K^{c}$ be the cyclotomic $\Zl$-extension of $K$. Let $H'_{n}$ be the $\ell$-field of $\ell$-classes of $K_{n}$, namely the maximal abelian $\ell$-extension of $K_{n}$ such that the places above $\ell$ are completely decomposed. Here $K_{n}$ is the unique subextension of $K^{c}$ with $[K_{n}:K]=\ell^{n}$. Class field theory identifies $\Gal(H'_{n}/K_{n})$ with the group $C\ell'_{n}$ of the $\ell$-classes of divisors of $K_{n}$. Let $H'_{\infty}=\bigcup H'_{n}$ and consider the Galois group $\mathcal{C}'=\Gal(H'_{\infty}/K^{c})$. Then $\mathcal{C}'$ is isomorphic to the projective limit $\varprojlim C\ell'_{n}$. Let $\Clog{n}$ be the logarithmic class group of $K_{n}$. Then for $n$ big enough we have 
$$\Clog{n}\isom \mathcal{C}'/\omega_{n}\mathcal{C}',$$
since $\omega_{n}\mathcal{C}'$ fixes the maximal abelian extension of $K_{n}$ which splits completely over $K_{\infty}$. 

The $\LL$-module $\mathcal{C}'$ is noetherian and torsion. Let $\mu'$ and $\lambda'$ be its structural invariants. Hence, Iwasawa theory yields that 
$$|\Clog{n}|=\ell^{\mu'\ell^{n}+\lambda' n + \tilde{\nu}}$$
for some $\tilde{\nu}\in\Z$ and $n$ big enough. Moreover, one can show that the $\mu$ invariant associated to the inverse limit $\varprojlim C\ell_{n}$ of the $\ell$-part of the class groups $C\ell_{n}$ of $K_{n}$ coincides with $\mu'$ and that $\lambda$ and $\lambda'$ differ just by some fixed factor $\lambda^{[\ell]}$ \cite[Prop. IV.2.5]{Jaulent86}, that is 
$$\mu=\mu'\	\	\	\text{ and }\	\	\	\lambda'=\lambda+\lambda^{[\ell]}.$$ 
 
\subsection{Topologies of $\Zl^{d}$-extensions}

Let $\Delta:=\Delta(K)$ be the set of all the $\Zl$-extensions of $K$. We will recall some topologies that can be defined on $\Delta$. 

\subsubsection{Greenberg's Topology}

Let $\Delta_{n}$ be the discrete topological space consisting of all cyclic extensions of $K$ of degree $\ell^{n}$ contained in some $\Zl$-extension $K_{\infty}$. We endow $\Delta$ with the topology induced by the inverse limit $\Delta = \varprojlim \Delta_{n}$ with the natural maps $\Delta_{m}\rightarrow \Delta_{n}$ for $m\geq n$. Hence $\Delta$ is a compact Hausdorff topological space with a basis of the topology given by
$$ \Delta(K_{\infty},n) := \{K'_{\infty} \in \Delta \,|\, [K_{\infty}\cap K'_{\infty} : K]\geq \ell^{n}\}$$
for $K_{\infty}\in \Delta$ and $n\in\N$.


\subsubsection{Kleine's Topology}\label{ss:ktop} Let $\Pl_{K}(\ell)$ be the set of primes of $K$ lying above $\ell$. For a $\Zl$-extension $K_{\infty}/K$ we define $P(K_{\infty })$ (resp. $\tilde{P}(K_{\infty})$) to be the subset of primes in $\Pl_{K}(\ell)$ that ramify (resp. log-ramify) in $K_{\infty}/K$. 

For every $\Zl$-extension $K_{\infty}$ of $K$ and $n\in \N$ we define
$$\Sigma(K_{\infty},n)=\{K'_{\infty}\in \Delta(K_{\infty},n)\,|\, P(K'_{\infty })\subseteq P(K_{\infty }) \}.$$
Analogously we define $\tilde{\Sigma}(K_{\infty},n)$. 

The sets $\Sigma(K_{\infty},n)$ (resp. $\tilde{\Sigma}(K_{\infty},n)$) generate a topology on $\Delta$, which we call Kleine's topology (resp. log-Kleine's topology).

\begin{remark} \begin{itemize}
\item[(1)] The set $\Delta(K)$ is not compact with Kleine's topology (resp. log-Kleine's topology).
\item[(2)] If the Gross-Kuz'min conjecture holds for every number field $L$ contained in $K_{d}$, then in the log-Kleine's topology the cyclotomic $\Zl$-extension $K^{c}$ of $K$ is eventually isolated from any other $\Zl$-extension of $K$ \cite[Lemma 2.6]{Villanueva18}. 
\end{itemize}

\end{remark}

\section{The logarithmic module}\label{Sec:LogMod}

Let $K_{d}$ be a $\Zl^{d}$-extension of $K$, let $L_{d}$ be the maximal abelian logarithmically unramified $\ell$-extension of $K_{d}$. Let $X_{d}=\Gal(L_{d}/K_{d})$ and denote $G$ the Galois group $\Gal(L_{d}/K)$. As usual we consider $X_{d}$ as a $\LL_{d}$-module.

Before proving our main result, let us recall the following.

\begin{lemma}\label{lem:Nakzld} If $N$ is a Galois pro-$\ell$-extension of $K$ such that $K_{d}\subseteq N$ and $\Gal(N/K_{d})$ is abelian, then $\Gal(N/K_{d})$ is a noetherian $\LL_{d}$-module if and only if $\Gal(N_{0}/K)$ is a noetherian $\Zl$-module, where $N_{0}$ denotes the maximal abelian extension of $K$ contained in $N$.
\end{lemma}
\begin{proof} See proof of \cite[Thm 1.]{Greenberg73IwasawaInvariants}.
\end{proof}

Now we state our main result.

\begin{theorem}\label{thm:Xlognoetors} 
Let $K_{d}$ be a $\Zl^{d}$-extension and $L_{d}$ as above. Then the $\LL_{d}$-module $X_{d}$ is noetherian. Moreover, if the Gross-Kuz'min conjecture holds for every subextension $K_{n}$ of some $\Zl$-extension $K_{\infty}$ then $X_{d}$ is $\LL_{d}$-torsion.
\end{theorem}
\begin{proof} By Lemma \ref{lem:logramlram} the maximal abelian $\ell$-extension $L_{d,0}$ of $K$ containing $K_{d}$ which is logarithmically unramified over $K_{d}$ is contained in $M_{0}$, the maximal abelian $\ell$-extension of $K$ which is $\ell$-ramified over $K_{d}$. Since $\Gal(M_{0}/K)$ is a noetherian $\Zl$-module. Then $\Gal(L_{d,0}/K)$ is also a noetherian $\Zl$-module. Lemma \ref{lem:Nakzld} implies that $\Gal(L_{d}/K_{d})$ is a noetherian $\LL_{d}$-module.

In order to prove the torsion property suppose $d\geq 2$, otherwise we know that the result is just Proposition 3.1 in \cite{Villanueva18} in the non-cyclotomic case and for the cyclotomic $\Zl$-extension this follows from Section \ref{sec:thecyclocase}.

Let $\p_{1},\ldots,\p_{s}$ be the primes of $K$ above $\ell$ that do not split completely in $K_{d}$. This set is non-empty by Lemma \ref{lem:logramZlext}. If $d=2$, it is easy to see that there exists a $\Zl$-extension $K_{\infty}$ of $K$ contained in $K_{d}$, in which none of the $\p_{1},\ldots,\p_{s}$ is completely decomposed. An inductive argument shows that one can always find such a $\Zl$-extension for $d\geq2$.

Let $G=\Gal(L_{d}/K_{\infty})$. We have the short exact sequence of $\Zl$-modules
$$0\rightarrow X_{d} \rightarrow G \rightarrow \Zl^{d-1}\rightarrow 0.$$
Taking homology with coefficients in $\Z$ we obtain the exact sequence
$$H_{2}(\Zl^{d-1},\Z)\rightarrow X_{d}/\mathfrak{A}X_{d}\rightarrow G/[G,G] \rightarrow \Zl^{d-1}\rightarrow 0.$$
Here $\mathfrak{A}$ is the kernel of the map $\Gal(K_{d}/K)\rightarrow \Gal(K_{\infty}/K)$. We have that $H_{2}(\Zl^{d-1},\Z)$ is trivial since $\Zl^{d-1}$ is free. It suffices to prove that $G/[G,G]$ is $\LL$-torsion, for if $X_{d}/\mathfrak{A}X_{d}$ is $\LL$-torsion then $X_{d}$ is $\LL_{d}$-torsion. 

Let $\tilde{I}$ be the subgroup of $G/[G,G]$ generated by the logarithmic inertia subgroups if all places of $K_{\infty}$ above $\p_{1},\ldots,\p_{s}$. Notice that there is a finite number of such logarithmic inertia subgroups. $\tilde{I}$ is of finite type over $\Zl$. Then we have the following exact sequence
$$0\rightarrow \tilde{I} \rightarrow G/[G,G] \rightarrow \widetilde{X} \rightarrow 0,$$
where $\widetilde{X}$ is the Galois group of a maximal abelian $\ell$-extension of $K_{\infty}$ which is logarithmically unramified. Hence $\widetilde{X}$ is a noetherian $\LL$-module, and if the Gross-Kuz'min conjecture holds in the finite layers $K_{n}$ of $K_{\infty}$, then is torsion.
\end{proof}

\section{Topological behavior of logarithmic invariants}\label{Sec:LGM}

\begin{definition} Let $K_{\infty}$ be a $\Zl$-extension of $K$. Recall that a place $\p$ of $K$ splits finitely in the tower $K_{\infty}/K$ if its decomposition group in $\Gal(K_{\infty}/K)$ is open. We say that $K_{\infty}$ splits finitely if very place $\p$ above $\ell$ splits finitely.
\end{definition}

Let $\Delta^{0}$ be the subset of $\Delta(K)$ consisting of the $\Zl$-extensions of $K$ that split finitely. We know that $\Delta^{0}$ is a dense subset of $\Delta(K)$ \cite[Prop. 3]{Greenberg73IwasawaInvariants}. 

The following theorem has been proven by Greenberg in the classical case. It has a local nature and its probably the precursor of the use of topological methods to deduce information on the structural invariants of torsion modules over Iwasawa algebras. 

\begin{theorem}\label{thm:mulaboun} Let $K_{\infty}$ be a $\Zl$-extension that splits finitely, i.e. $K_{\infty}$ is in $\Delta^{0}$. Let $\tilde{\mu}(K_{\infty})$ and $\tilde{\lambda}(K_{\infty})$ be the structural invariants of its logarithmic module $X$. Then with respect to Greenberg's topology:
\begin{itemize}
\item[(i)] The invariant $\tilde{\mu}$ is bounded in a neighborhood of $K_{\infty}$.
\item[(ii)] If $\tilde{\mu}(K_{\infty})=0$ then the invariants $\tilde{\mu}$ and the $\tilde{\lambda}$ are respectively zero and bounded in a neighborhood of $K_{\infty}$. 
\end{itemize}
\end{theorem}
\begin{proof}
(i) follows directly from Theorem \ref{thm:Xlognoetors}  and the fact that the classical invariant $\mu$ and its logarithmic conterpart coincide (see \S 2.5 and \textbf{\cite[Thm. 4.8]{Villanueva18}}). For (ii) the first part follows for the same reasons commented before. However to show that $\tilde{\lambda}$ is bounded we must proceed as in the proof of \cite[Thm. 3]{Greenberg73IwasawaInvariants}. \end{proof}

In recent years, S\"oren Kleine proved in \cite{Kleine17} a version of \ref{thm:mulaboun} using a different approach from that of Greenberg. In particular two things are essential in his results: the topology described in section \ref{ss:ktop} and the generalization of a result of Fukuda. Despite of the fact that Kleine's topology is finer than Greenberg's topology, the main advantage of Kleine's technique is to get rid of the splitting assumption, namely the restriction to $\Delta^{0}$.

It is natural to consider Kleine's topology adapted to logarithmic arithmetic, i.e. taking into account logarithmic ramification. Indeed, a straightforward rephrase of Kleine's results gives us.

\begin{theorem}\label{thm:mulabounktop} Let $K_{\infty}$ be a $\Zl$-extension. Let $\tilde{\mu}(K_{\infty})$ and $\tilde{\lambda}(K_{\infty})$ be the structural invariants of its logarithmic module $X$. Then with respect to log-Kleine's topology:
\begin{itemize}
\item[(i)] The invariant $\tilde{\mu}$ is bounded in a neighborhood of $K_{\infty}$.
\item[(ii)] If $\tilde{\mu}(K_{\infty})=0$ then the invariants $\tilde{\mu}$ and the $\tilde{\lambda}$ are respectively zero and bounded in a neighborhood of $K_{\infty}$. 
\end{itemize}
\end{theorem}

The next question, raised by Greenberg, is whether there is a global bound of such invariants. The answer has been positively answered in the classical case independently by Baba\u{\i}cev \cite{Babauicev80} and Monsky \cite{Monsky81SomeInvariantsZpdExt}.  


\begin{theorem} Let $K$ be a number field. Assume the Gross-Kuz'min conjecture as in Theorem \ref{thm:Xlognoetors}. Then $\tilde{\mu}(K_{\infty})$ is bounded for every $K_{\infty}\in \Delta(K)$. 
\end{theorem}
\begin{proof} The result is clear when we restrict to $\Zl$-extensions, i.e. the case $d=1$. Now suppose the result holds for every $\Zl^{i}$-extension $M$ of $K$ with $i<d$. Let $K_{d}$ be an arbitrary $\Zl^{d}$-extension and let $X_{d}$ be the $\LL_{d}$-module $\Gal(L_{d}/K_{d})$ corresponding to $L_{d}$ the maximal abelian logarithmically unramified $\ell$-extension of $K_{d}$. As we have seen, $X_{d}$ is noetherian and assuming the Gross-Kuz'min conjecture it is torsion. We denote $\Delta(M/K)$ the set of $\Zl$-extensions of $K$ contained in $M$. Let $K_{\infty}\in \Delta(K_{d}/K)$. 

Assume that $K_{\infty}$ splits finitely. By Theorem 4.8 in \cite{Villanueva18} the classical $\mu$ invariant coincides with the logarithmic $\mut$ invariant. By Theorem 3.3 in \cite{Babauicev80}, we know that the classical $\mu$ invariants are bounded in $\Delta(K)$.

Let $\{\p_{1},\ldots,\p_{s}\}$ be the set of primes of $K$ above $\ell$ ramifying logarithmically in $K_{d}$. Assume that at least one of the $\p_{i}$, say $\p_{1}$ splits completely in some $\Zl$-extension $K_{\infty}$, otherwise $K_{\infty}$ would split finitely. Then the decomposition group $D_{\p_{1}}\subset \Gal(K_{d}/K)$ is such that
$$K_{\infty} \subseteq K_{d}^{D_{\p_{1}}}\	\	\	\text{ and }\	\	\	\Gal(K_{d}^{D_{\p_{1}}}/K)\isom \Zl^{d_{1}},$$
for some $d_{1}<d$ since $\p_{1}$ cannot split completely in $K_{d}/K$. Any other $\Zl$-extension $L$ of $\Delta(K)$, should be contained in one of the $K_{d}^{D_{\p_{i}}}$. By induction it follows that the $\mut$ invariant is bounded there as well. \end{proof}
  
In the proof of the preceding theorem we make use of the fact that the $\mut$ invariant equals its classic counterpart $\mu$ whenever $K_{\infty}$ is in $\Delta^{0}$. The proof of the theorem would be straightforward whenever this equality holds in general. Therefore it is very natural to ask whether this is the case. This is source of future work.

\bibliographystyle{amsplain}
\bibliography{Biblio}  

\providecommand{\bysame}{\leavevmode\hbox to3em{\hrulefill}\thinspace}
\providecommand{\MR}{\relax\ifhmode\unskip\space\fi MR }
\providecommand{\MRhref}[2]{%
  \href{http://www.ams.org/mathscinet-getitem?mr=#1}{#2}
}
\providecommand{\href}[2]{#2}
\begin{thebibliography}{10}

\bibitem{Babauicev80}
V.~A. Baba\u\i~cev, \emph{On the boundedness of {I}wasawa's {$\mu
  $}-invariant}, Izv. Akad. Nauk SSSR Ser. Mat. \textbf{44} (1980), no.~1,
  3--23, 238. \MR{563783}

\bibitem{Belabas&Jaulent16}
K.~Belabas and J.-F. Jaulent, \emph{{The logarithmic class group package in
  PARI/GP}}, Pub. Math. Besan\c{c}on (2016).

\bibitem{Greenberg73IwasawaInvariants}
Ralph Greenberg, \emph{{The {I}wasawa invariants of {${\bf \Gamma
  }$}-extensions of a fixed number field}}, Amer. J. Math. \textbf{95} (1973),
  204--214. \MR{0332712}

\bibitem{Jaulent86}
J.-F. Jaulent, \emph{L'arithm\'etique des {$l$}-extensions}, Publications
  Math\'ematiques de la Facult\'e des Sciences de Besan\c con., Universit\'e de
  Franche-Comt\'e, Facult\'e des Sciences, Besan\c con, 1986, Dissertation.
  \MR{859709 (88j:11080)}

\bibitem{Jaulent94}
\bysame, \emph{Classes logarithmiques des corps de nombres}, J. Th\'eor.
  Nombres Bordeaux \textbf{6} (1994), no.~2, 301--325. \MR{1360648 (96m:11097)}

\bibitem{Kleine17}
S\"{o}ren Kleine, \emph{Local behavior of {I}wasawa's invariants}, Int. J.
  Number Theory \textbf{13} (2017), no.~4, 1013--1036. \MR{3627696}

\bibitem{Monsky81SomeInvariantsZpdExt}
Paul Monsky, \emph{{Some invariants of {${\bf Z}^{d}_{p}$}-extensions}}, Math.
  Ann. \textbf{255} (1981), no.~2, 229--233. \MR{614399}

\bibitem{Neukirch&Schmidt&Wingberg08}
J.~Neukirch, A.~Schmidt, and K.~Wingberg, \emph{{Cohomology of number fields}},
  second ed., {Grundlehren der Mathematischen Wissenschaften}, vol. 323,
  Springer-Verlag, Berlin, 2008. \MR{2392026}

\bibitem{Villanueva18}
J.~I. {Villanueva-Guti\'errez}, \emph{{On the mu and lambda invariants of the
  logarithmic class group}}, ArXiv e-prints. (2018), arXiv:1802.04006
  [math.NT].

\bibitem{Washington97}
L.~C. Washington, \emph{{Introduction to cyclotomic fields}}, second ed.,
  {Graduate Texts in Mathematics}, vol.~83, Springer-Verlag, New York, 1997.
  \MR{1421575 (97h:11130)}

\end{thebibliography}
\end{document}